\documentclass[12pt]{amsart}
\usepackage{geometry}                
\usepackage{graphicx}

\usepackage{amsmath}
\usepackage{amsfonts}
\usepackage{enumerate}
\usepackage{amssymb}

\usepackage{epstopdf}
\usepackage{hyperref}

\newtheorem{theorem}{Theorem}[section]

\newtheorem{lemma}[theorem]{Lemma}

\theoremstyle{definition}

\newtheorem{remark}[theorem]{Remark}

\numberwithin{equation}{section}

\begin{document}

\title{Surfaces of small diameter with large width}

\author[Y.~Liokumovich]{Yevgeny Liokumovich}
\address {Department of Mathematics, University of Toronto, Toronto, Canada}
\email{liokumovich@math.toronto.edu}

\begin{abstract}
Given a 2-dimensional surface $M$ and a constant $C$
we construct a Riemannian metric $g$, so that
 diameter $diam(M,g) \leq 1$ and every $1-$cycle
 dividing $M$ into two regions of equal area has length $> C$.
 It follows that there exists no universal inequality 
 bounding $1-$width of $M$ in terms of its diameter.
 This answers a question of St\'ephane Sabourau.
\end{abstract}

\date{}

\maketitle

\section{Introduction}

In this paper we study the relationship between isoperimetric profile of 
a Riemannian closed 2-surface $M$ and its diameter $d$. Let $g$ denote the genus of
$M$. The first result of this paper is

\begin{theorem} \label{subdivision}
For every $\epsilon \in (0,\frac{1}{2})$ there exists a constant $C(\epsilon) \leq \epsilon ^{-\frac{1}{\log_2(3/2)}}$, such that
for any Riemannian metric on $M$ one can find a simple closed curve of length
$\leq (C(\epsilon)+ 8 g) d$ subdividing $M$ into two regions, each of area 
$\geq (\frac{1}{2}-\epsilon) Area(M)$.
\end{theorem}

For a non-orientable surface the genus $g$ in the theorem can be replaced by a smaller 
quantity $\frac{1}{2} rank H_1(M, \mathbb{Z}_2 )$.
This is a generalization of Theorem 6.2 in \cite{LNR}, where the result is proven for 2-spheres.

In this paper we show that the optimal constant in
Theorem \ref{subdivision} goes to infinity as
$\epsilon \rightarrow 0$. 

This is the case even if we use a collection of closed curves 
to subdivide $M$. We will formulate the result for Lipschitz 1-cycles,
which generalize finite combinations of closed curves. 
For example, they include curves traversed multiple times.
We will construct Riemannian closed surfaces
 of small diameter with the following property:
if a minimal filling of a 1-cycle has area close to half of the
area of the surface, then the cycle must be very long. 
As a special case this implies that a collection of 
simple closed curves subdividing $M$ into two parts of equal area
must have large total length.

Let $\Bbbk$ be an abelian group.
Recall that Lipschitz $k-$chains with coefficients in $\Bbbk$ are finite sums 
$\sum a_i f_i$, where $a_i \in \Bbbk$ and $f_i$
are Lipschitz maps from the $k-$simplex $\Delta ^k$ to $M$. 
The volume of a $k-$chain $c=\sum a_i f_i$
denoted by $|c|$ is defined to be $\sum |a_i| \int _{\Delta^k} f_i ^{*} g$. The boundary map $\delta$
is defined as in the singular homology theory. 

Throughout this paper we will assume that $\Bbbk = \mathbb{Z}$ 
if $M$ is orientable and $\Bbbk= \mathbb{Z}_2$ otherwise.

Given a Lipschitz $k-$cycle $c$ we define its filling volume
$Fill(c) = \inf \{ |A|\}$, where the infimum is taken over all $k+1$
Lipschitz chains $A$ with $\partial A =c$. 

\begin{theorem} \label{main}
For every closed surface $M$ and every $C>0$ there exists 
a Riemannian metric $g$ on $M$, 
such that diameter of $(M,g)$ equals $1$
and every Lipschitz 1-cycle $c$
with $Fill(c)= \frac{1}{2}Area(M)$
has length $|c|\geq C$.
\end{theorem}

The flat pseudo-norm of a $k-$chain $c$ is defined as the infimum of $|c - \partial A|+|A| $ 
over all Lipschitz $(k+1)-$chains $A$. We identify two chains if
their difference has flat pseudo-norm equal to 0. The completion of this space
is called the space of integral flat $k-$chains and will be denoted by $Z_k(M, \Bbbk)$. 
For details of construction of $Z_k(M, \Bbbk)$ and its properties see a paper of W. Fleming
\cite{Fl}.

The topology of the space of cycles has been studied by F. Almgren in the 60s.
In his PhD thesis he established an isomorphism between the
$m-$th homotopy group of the space of flat $k-$cycles $\pi_m(Z_k(M, \Bbbk))$
and $(m+k)-$th homology group $H_{m+k}(M; \Bbbk)$ of $M$ \cite{Al}. 
A family of cycles that corresponds to a non-trivial element of homology
of $M$ is called a sweep-out of $M$ and the volume of the largest cycle in the 
sweep-out is called the width of the sweep-out.
The $k-width$ of $M$, $W_k(M)$, is defined to be the infimum over 
widths of all sweep-outs of $M$ by $k-$cycles.
When $k=1$, $W_1(M)$ is sometimes called a $diastole$ of $M$ (\cite{BS}).
$W_k(M)$ is an important geometric invariant of a Riemannian manifold.
Almgren proved that every manifold admits a stationary $k-$cycle of
volume $\leq W_k(M)$ and J. Pitts obtained regularity results for this
cycle (see \cite{P}).

Various inequalities relating $W_k(M)$ to other geometric invariants of $M$
exist in the literature. L. Guth in \cite{G} proved that for every bounded
open subset $U$ of $\mathbb{R}^n$ $W_k(U)\leq C(n) |U|^{\frac{k}{n}}$, where 
$C(n)$ depends on the dimension $n$ only.

F. Balacheff and S. Sabourau proved in \cite{BS} that for any Riemannian
surface of genus $g$, $W_1(M) \leq  C \sqrt{g+1} \sqrt{|M|} $ for a
universal constant $C$.

The author learnt from A. Nabutovsky of the following question asked by S. Sabourau
in a private conversation. Is it possible to construct metrics on 
a Riemannian 2-sphere of bounded diameter and
 arbitrarily large 1-width? Previously it has been 
 shown in \cite{L} that the width of an ``optimal" sweep-out of a 2-sphere by closed curves,
 rather than cycles, can not be controlled in terms of diameter.
 We generalize the result in \cite{L} and answer the question of Sabourau positively.

\begin{theorem} \label{width}
For any smooth closed 2-surface $M$ and any $C>0$ there exists a 
Riemannian metric on $M$ such that $diam(M) = 1$ and
$W_1(M) \geq C$.
\end{theorem}

The construction relies on the idea of S. Frankel 
and M. Katz in \cite{FK} of defining metrics using embeddings of trees.

\begin{figure} 
   \centering	
	\includegraphics[scale=1]{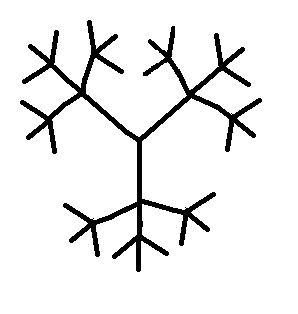}
	\caption{Ternary tree} \label{ternary_fig}
\end{figure}

We embed a ternary tree $T$ (see Figure \ref{ternary_fig}) into the manifold and define a metric on $M$, such that

\begin{enumerate}
\item
the diameter of $M$ is $\leq 1$;

\item
the area of $M$ is mostly concentrated in a small neighbourhood of $T$;

\item
the distance between non-adjacent edges of $T$ is bounded from below by some positive constant.
\end{enumerate}

Then one observes that to subdivide a ternary tree into two equal parts by a collection of closed curves,
one either has to have one closed curve that intersects many edges or many short 
closed curves or some combination of both.

One concludes that the total length of a 1-cycle dividing $M$ into two parts 
of equal area must be large.

We first prove the result for a Riemannian metric on the 2-sphere $M=(S^2,g)$ and then
extend it to general surfaces.

\medskip\noindent
{\bf Acknowledgements.}
The author would like to thank Alexander Nabutovsky, Regina Rotman and Florent Balacheff 
for valuable discussions of this paper. The author is thankful to the anonymous referee
for many suggestions that helped to improve the exposition.

The author was supported by Natural Sciences and Engineering Research
Council (NSERC) CGS scholarship.

\section{Upper bound on the length of a subdividing curve}

In this section we prove Theorem \ref{subdivision}.

Let $n$ denote the number of generators 
in a minimal basis for homology group $H_1(M;\Bbbk)$.

Let $a \in M$ be any point. There exists a collection $\Gamma = \{\gamma_1,...,
\gamma_n \}$
of simple loops based at $a$ 
 that form a basis for homology $H_1(M, \Bbbk)$. Furthermore, they can be chosen so that
 they do not pairwise intersect except
at the point $a$, and $|\gamma_i| \leq 2d$. Indeed, following an argument of M.Gromov in 
\cite{Gro}
 we can construct them inductively as follows. Let $\gamma_1$ be the shortest homologically
  non-trivial loop based at $a$ and define $\gamma_k$ to be the shortest loop based at $a$ 
  that is homologically independent from $\{ \gamma_1,...,\gamma_{k-1} \}$. 
Choose a constant speed parametrization of $\gamma_k$. Then two arcs of $\gamma_k$
from $a$ to the point $b=\gamma_k(\frac{1}{2})$ are minimizing geodesics.
Moreover, these are the only minimizing geodesics from $a$ to $b$.
 For suppose
there is a different path $\alpha$ from $a$ to $b$ of length $\leq \frac{1}{2} |\gamma_k|$.
Set $\beta_1 = \alpha + \gamma_k|_{[\frac{1}{2},1]}$ 
and $\beta_2 = \gamma_k|_{[0,\frac{1}{2}]}- \alpha$.
We observe that by smoothing out the corner of $\beta_i$
at the point $b$ we can make it strictly shorter than $\gamma_k$.
It follows that $\beta_1$ and $\beta_2$ are generated by 
$\{ \gamma_1,...,\gamma_{k-1} \}$ and hence so is $\gamma_k$. 
This is a contradiction. 
 
Since each loop is length minimizing until the midpoint 
and there are no other length minimizing geodesics from $a$ to the midpoint,
it follows that two loops do not intersect except at $a$.

We cut surface $M$ along the curves of $\Gamma$ to obtain a manifold with boundary $D$.
Since curves of $\Gamma$ are homologically independent, successively cutting along them does not separate the surface. Hence, $D$ is connected. We claim that it is homeomorphic to a disc.
Let $U$ be a small tubular neighbourhood of the union of curves in $\Gamma$.
Consider Mayer-Vietoris sequence for the decomposition $M=U \cup D$.
Map $H_1(U \cap D; \Bbbk) \rightarrow H_1(U; \Bbbk) \oplus H_1(D; \Bbbk)$
sends the generator of homology of $U \cap D \simeq S^1$ to $0$.
Map $H_1(U; \Bbbk) \oplus H_1(D; \Bbbk) \rightarrow H_1(M; \Bbbk)$ is an isomorphism
when restricted to $H_1(U; \Bbbk)$. Hence, $ H_1(D; \Bbbk) = 0$ and it follows
that $D$ is a disc. 
The length of the boundary of $D$ satisfies $|\partial D| \leq 4 n d$.

By Theorem 6.1 B in \cite{LNR} for every $\delta > 0$ there exists a curve $\beta$ of
length $\leq 2 \sup _{x \in D} dist(x, \partial D) + \delta \leq 2 d + \delta$
 with endpoints on the boundary of $\partial D$, which does
not self-intersect and divides $D$ into subdiscs $D_1$ and $D_2$ satisfying
$\frac{1}{3}Area(D) - \delta^2 \leq Area(D_i ) \leq \frac{2}{3} Area(D) + \delta^2$.
We keep applying the theorem to subdivide each subdisc into smaller subdiscs.
After $ N= \lfloor \log_{\frac{2}{3}+o(\epsilon)} (2 \epsilon) \rfloor +1$ iterations we subdivide 
$D$ into $2^N =  (2 \epsilon) ^{-\frac{1}{\log_2(3/2)}} + o(\delta)$
subdiscs of area less than $2 \epsilon A(D)$. By choosing $\delta$ sufficiently small
we can make $2^N < \epsilon ^{-\frac{1}{\log_2(3/2)}}$.

It is well-known that every cell subdivision of a $2$-disc into a finite
collection of cells is shellable, that is, we can enumerate these $N$ discs
by integers $1,..., N$ so that for every positive integer $k<N$ the union
of the first $k$ discs is homeomorphic to a disc. 
Hence, we can start adjoining discs one by one until we obtain a subdisc of area
within $\epsilon$ of $\frac{1}{2} Area(M)$. Then the length of the boundary 
of this disc 
will be bounded by $(2^N+4n) d + o(\delta) < (\epsilon ^{-\frac{1}{\log_2(3/2)}} +8 g)d $
for $\delta$ sufficiently small.

\section{Spheres that are hard to subdivide}

We prove Theorem \ref{main} in the case when $M=(S^2,g)$.

\textbf{1. Embedding of a ternary tree in $M$}

We use a construction similar to that of S. Frankel and M. Katz \cite{FK}.

The difference with construction in \cite{FK} is that we use a ternary tree instead of a binary tree.
Define a ternary tree of height $n$ inductively as follows.
Ternary tree of height $0$ is a point. A ternary tree of height $n$ is obtained by 
taking a ternary tree of height $n-1$ and attaching 3 new edges to each vertex that 
has less than 3 edges. Let the function $N(h)= \frac{3}{2}(3^h - 1)$ 
denote the number of edges in a ternary tree of height $h$.

Given a ternary tree $T$ of height $h$ we will define an embedding $\iota: T \rightarrow S^2$
and a metric $g$ on $S^2$, such that diameter of $(S^2,g)$ is $\leq 1$ and the distance between 
images of non-adjacent edges is $\geq 3/4$. We do it in the following way.

Consider hyperbolic plane $\mathbb{H}^2$ of constant negative curvature $-K^2$.
Let $D$ be a closed disc of radius $1/2$ in $\mathbb{H}^2$.
Fix a large constant $h \in \mathbb{N}$.
Subdivide the boundary of $D$ into $2N(h)$ equal intervals. 
We require that $K$ is large enough so that the following holds:

\begin{enumerate}
\item
the distance between the endpoints of an interval, denoted by $E$, is $\geq 3/4$;

\item
all but $\frac{1}{N(h)}$ of area of $D$ is contained within $\frac{1}{4}$-neighbourhood of the boundary.
\end{enumerate}

We define a quotient map $q$ on $\partial D$ that identifies pairs of intervals
in such a way that $q(\partial D)$ is a ternary tree $T$ with $N(h)$ edges.
In other words, we consider a closed tour of the tree so that each edge
is traversed twice, and regard this tour as a map $q:S^1 = \partial D \rightarrow T$.
We then extend $q$ to a map of $D$ to a sphere $M$ so that it is an isometry on the interior
of $D$.
We can think of it as ``gluing" the boundary of $D$ onto a ternary tree $T$
to obtain a sphere $M$ with singular metric on $T$ and hyperbolic metric 
on the open set $M \setminus T$. 

\begin{remark}
Alternatively, as was suggested to the author by F. Balacheff,
one could consider a $2N(h)-$gon made out of $2N(h)$ flat equilateral triangles with a
common vertex. Edges of the $2N(h)-$gon are pairwise identified to obtain 
a ternary tree in the same way we did it above for the hyperbolic disc. 
We obtain a singular metric on the sphere satisfying properties (1)-(3)
from the introduction. If we glue triangles adjacent to the same 
edge of the ternary tree so that they make two sides of ``one triangle", we can visualize 
this example as in Figure \ref{cone}.

The difference between these two constructions is that negative curvature
in the first case is uniformly distributed over
the complement of tree $T$ and in the second case it is concentrated at a point
(the central vertex). To prove Theorem \ref{main} we will use the first construction
with hyperbolic disc. The argument can be adopted to prove the same result
for the sphere on Figure \ref{cone}.
\end{remark}

\begin{figure} 
   \centering	
	\includegraphics[scale=0.3]{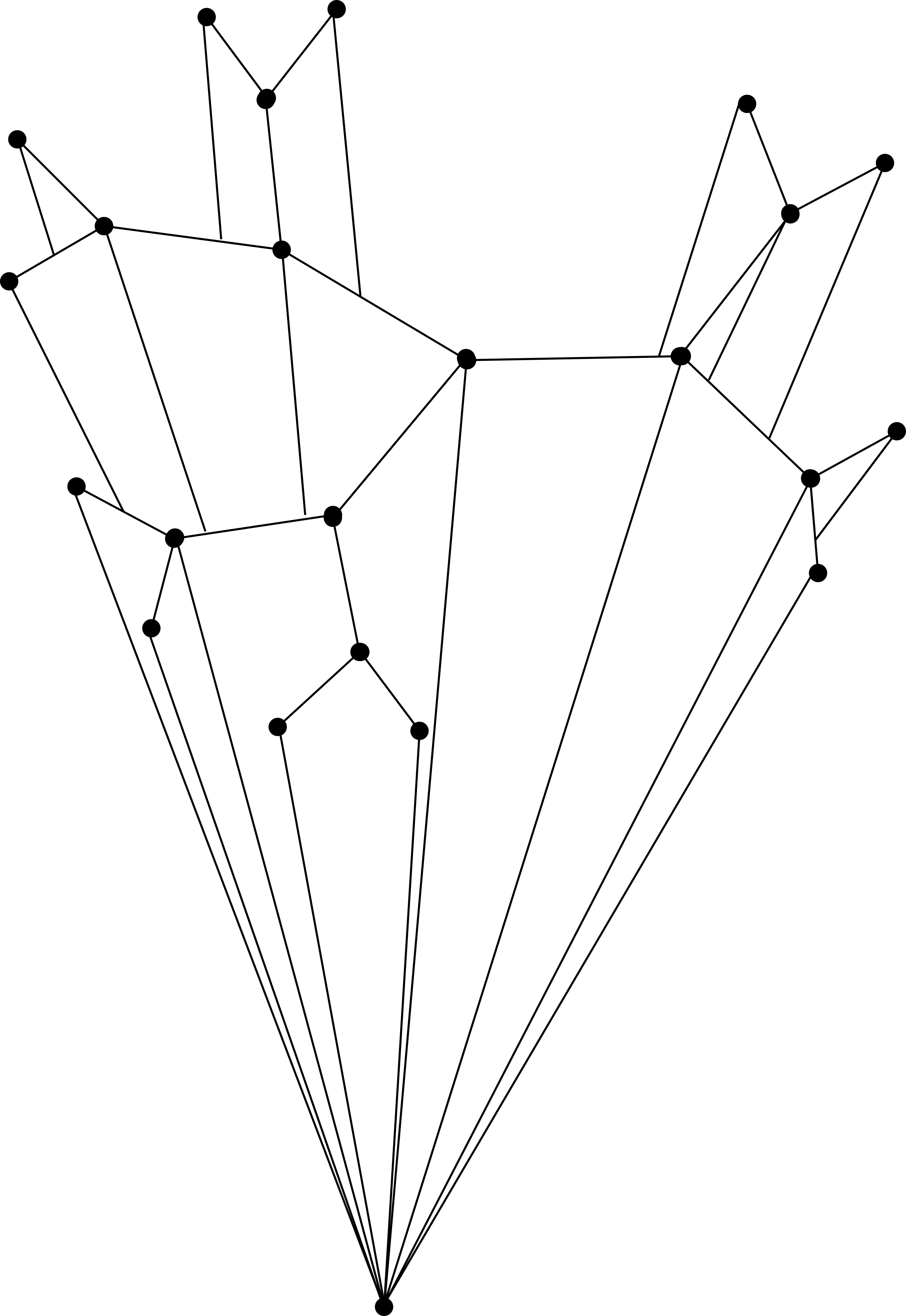}
	\caption{Sphere that is hard to subdivide. Figure courtesy of F. Balacheff.} \label{cone}
\end{figure}

The main fact about hyperbolic geometry that we will be using
 is the isoperimetric inequality in $\mathbb{H}^2$.

\begin{theorem} \label{isoperimetric}
Let $\mathbb{H} ^2$ be the 2-dimensional hyperbolic space of constant curvature $-K^2$.
Then for every integral flat $1-$cycle $c$ we have $|c| \geq K Fill(c)$
\end{theorem}

For the proof see \cite{BZ}. Isoperimetric inequality holds for cycles with coefficients in $\mathbb{Z}$ or $\mathbb{Z}_2$.

\textbf{2. Plan of the proof}

Let $c$ be a cycle whose minimal filling $A$ has area close to half 
of the volume of $M$. We will prove that if height $h$ of the embedded
ternary tree is large enough then $|c| \geq \frac{h}{3}$.

We will successively approximate cycle $c$ by $1-$cycles with certain special 
properties that will allow us to relate its length and the area enclosed by it to
the way it intersects the ternary tree. 
In the process of approximation the length and filling area of the
cycle will change, but in a controlled way.

In section 3 we will replace all arcs of $c \cap M \setminus T$ with geodesic
arcs in the hyperbolic disc $M\setminus T$. The 2-chain filling 
this new cycle we will denote by $A_1$. From the isoperimetric inequality 
in hyperbolic plane it will follow that the difference $|Fill(c) - |A_1||$
can be bounded it terms of length of the cycle $c$ (see (\ref{A_1})).

Our second approximation (section 4) will be to erase all arcs of $\partial A_1$
that have length $<E$ (distance between two adjacent vertices of the tree).
Again the area will change by an amount controlled by $|c|$ as in (\ref{A_2}).
The resulting 2-chain we call $A_2$.

In section 5 we show that the area enclosed by a cycle is 
proportional (up to a small error) to the number of edges of tree $T$ enclosed by it.

In section 6 we give a combinatorial argument based on simple Lemma \ref{sum}
showing that if $|\frac{|M|}{2}-|A_2||$ is small, then the sum of total number
of whole edges of $T$ in $ \partial A_2 \cap T$ and arcs of  $\partial A_2 \setminus T$
must be large. By inequality (\ref{A1A2}) this implies that the length of $\partial A_1$
is large. Combining this with inequalities (\ref{A_1}) and (\ref{A_2}) we
conclude that the length of $c$ must be large.

\textbf{3. Approximating $|c|$ by images of geodesic arcs}

Let $A$ be a Lipschitz chain in $M$ filling $c$
with $||A|- \frac{1}{2} |M|| < \epsilon$ for some small positive $\epsilon< \frac{|M|}{N}$.
It is a well-known result in geometric measure theory that 
an integral Lipschitz chain can be approximated by a chain $A_0=\sum a_i f_i$
where each $f_i$ is smooth and $|A|-|A_0|$ and $|\partial A| - |\partial A_0|$
are arbitrarily small (see \cite{Fe}, 4.2.20). 
We make another small perturbation to 
ensure that $\partial A_0$ intersects the tree $T$ transversally and does not intersect
vertices of $T$. Further we may assume that $\partial A_0 = \sum c_i$, where each $c_i$
is either a closed curve in $M \setminus T$ or an arc with endpoints on $T$ whose interior 
avoids the tree $T$.

\begin{figure} 
   \centering	
	\includegraphics[scale=0.6]{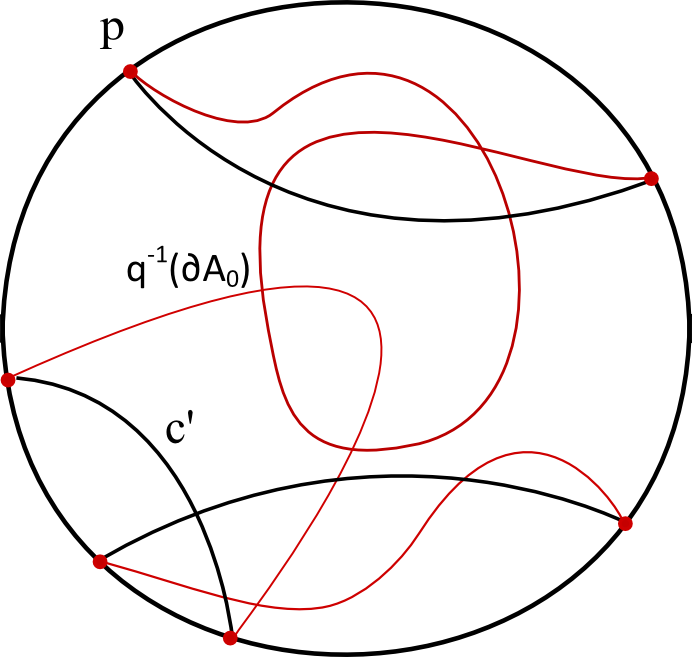}
	\caption{Constructing new cycle $\partial A_1$} \label{geodesic_arcs}
\end{figure}

We construct a new 2-chain $A_1$ as follows.
For each cycle $c_i \in \partial A_0$ that avoids the tree $T$ we can fill
it by a 2-cycle of area $\leq \frac{2}{K} |c_i|$.

Let $c_i$ be an arc with endpoints on $T$. We replace it with a geodesic arc in $M \setminus T$
in the following way.
Observe that the pre-image $q^{-1}(\partial A_0)$ is a $1-$chain in 
$D \subset \mathbb{H}^2$ with $p= \partial q^{-1}(\partial A_0)$ a $0-$cycle
in $\partial D$ (see Figure \ref{geodesic_arcs}). 
We fill $p$ in $D$ by a minimizing $1-$cycle $c'$ which consists of 
geodesic arcs in $D$. Then $\partial c' - q^{-1}(\partial A_0)$ is a cycle of 
length $\leq 2 |\partial A_0|$ (since $c'$ is minimizing).
Hence, we can fill it by a 2-chain $B$ of area $\leq \frac{4}{K} |\partial A_0|$.
Define $A_1=q(q^{-1}(A_0)+B)$. Then $\partial A_1$ is a union of arcs whose pre-images
under the quotient map are geodesics with endpoints on $\partial D$.

We have the following relationship between $A_1$ and the original cycle $c$:

\begin{eqnarray} \label{A_1}
|c| \geq |\partial A_1|, \nonumber \\ 
|Fill(c) - A_1| \leq \frac{4}{K} |c|
\end{eqnarray}

\textbf{4. Replacing short arcs by arcs in $T$}

Recall that by construction $M \setminus T$ is isometric
to an open disc of radius $\frac{1}{2}$ in the hyperbolic plane.
Each arc of $\partial A_1 \setminus T$ separates $M \setminus T$ into two subdiscs.
Let $D'$ denote the smaller of subdics.
Suppose the length of an arc $l$ is smaller than $E$, the distance between neighbouring vertices of $T$. 
We will call such an arc a short arc. Then 
we show that the area of $D'$ satisfies the following inequality.

\begin{equation} \label{short arcs}
|D'| \leq \frac{|M|}{N(h)} |l|
\end{equation}

\begin{figure} 
   \centering	
   \includegraphics[scale=0.6]{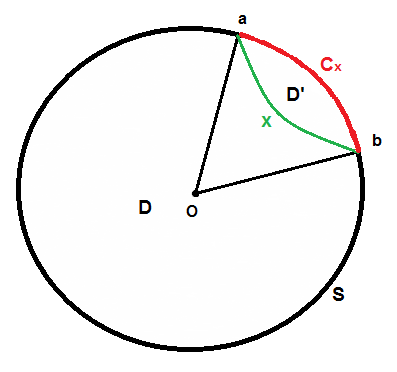}
   \caption{Short arc of length $x$ on a hyperbolic disc $D$} \label{shortarcf}
\end{figure}

Indeed, let $a$ and $b$ be two points on unit circle $S$ in hyperbolic space, let $x=dist_{\mathbb{H}^2}(a,b)$ and $C_x$
denote the length of the shorter arc of $S$ from $a$ to $b$
(see Figure \ref{shortarcf}).
Using hyperbolic trigonometry we observe that $\frac{x}{C_x}$ is a decreasing function of $x$.
Let $O$ denote the center of the hyperbolic disc. 
We observe that the area of a sector of the disc between $Oa$ and $Ob$
is proportional to $C_x$. If the length of an arc is equal to $C_E$ then the area of 
the corresponding sector is $\frac{|M|}{2 N(h)}$.
Therefore the area of the 
sector of $D$ containing $q^{-1}(D')$ is 
$\frac{C_{|l|}}{C_E} \frac{|M|}{2N(h)} \leq \frac{|l|}{E} \frac{|M|}{2N(h)}$.
Since $E\geq \frac{3}{4}$ we obtain that $|D'| \leq \frac{4}{3} \frac{|M|}{2N(h)} |l|$.

We replace each short arc $c_i$ of $\partial A_1$ by the shortest path in 
$T$ with the same endpoints and form this way a new 2-chain $A_2$.

The boundary $\partial A_2$ then consists of long arcs of $\partial A_1$
connected by paths in $T$. We contract each path in $T$ to the shortest path
with the same endpoints. Clearly, this does not change the area of $A_2$.

Let $L_s$ denote the union of all short arcs in $\partial A_1$.
By (\ref{short arcs}) we obtain that the area of $A_2$ satisfies

\begin{equation} \label{A_2}
|A_1| -  \frac{|M|}{N(h)} |L_s| \leq |A_2| \leq |A_1| +  \frac{|M|}{N(h)} |L_s|
\end{equation}

We will also need an estimate relating the number of whole edges in $\partial A_2 \cap T$
to the length of $L_s$. Let $l$ be a path in $\partial A_1$ with endpoints in $T$ consisting entirely of short 
arcs. Suppose the shortest path in $T$ connecting the endpoints of $l$ contains $n$ whole edges.
As noted above the ratio $\frac{x}{C_x}$ of the length $x$ of a geodesic path between two points 
on a circle in hyperbolic plane to the length $C_x$ of the circular arc that connects them is
a decreasing function of $x$. From this fact it is not hard to show that $|l| \geq n E$.

Let $m_l$ be the number of arcs in $\partial A_2 \setminus T$.
These are arcs of length $\geq E \geq \frac{3}{4}$ that lie in $\partial A_1 \cap \partial A_2$.
Let $m_s$ be the number of whole edges in $\partial A_2  \cap T$.
We obtain that the length of $\partial A_1$ satisfies 

\begin{equation} \label{A1A2}
|\partial A_1| \geq \frac{3}{4} (m_l + m_s)
\end{equation}


\textbf{5. Area and the number of edges enclosed by a cycle}

In the estimates below we do not want to keep track of the 
exact value of small terms that will be shown negligible in the end.
We will use the following notation. 
We will write $O(x)$ to denote a quantity that is between $-Cx$ and $Cx$
for some constant $C$ independent of $N$.

Suppose $l$ is a long arc, i.e. $|l| > E$. Let $e_1$ and $e_2$ be edges of $\partial D \subset \mathbb{H}^2$ connected by $q^{-1}(l)$
and $D'$ be one of the two discs in $D \setminus q^{-1}(l)$. Let $N'$ denote the number of whole
edges of $\partial D $ whose interiors are contained in
$\partial D'$. Recall that the distance $E'$ between two non-adjacent edges $e_1$ and $e_2$ satisfies $\frac{3}{4} < E \leq E' \leq 1$.
Geodesic $l$ does not enter $\frac{1}{4}-$neighbourhood of any edge other than $e_1$, $e_2$
and possibly one of the neighbours of each $e_1$ and $e_2$. 
Indeed, if there is a path $\alpha$ from a point on $l$
to an edge non-adjacent to $e_1$ or $e_2$ then by triangle inequality $|l|+ 2 |\alpha|> \frac{3}{2}$, so 
as $|l| \leq 1$ we have $\alpha > \frac{1}{4}$.
As all but at most $\frac{|M|}{N(h)} $ of area is concentrated 
in the $1/4-$neighbourhood of the boundary we conclude that the area of $D'$ is given by

\begin{equation} \label{long arc}
|D'|= \frac{|M| N'}{2N(h)} + O(\frac{|M|}{N(h)})
\end{equation}

\textbf{6. Combinatorial estimate}

Let $m = 2( m_l + m_s)$.
By the inequality (\ref{A1A2}) we $|\partial A_1| \geq \frac{3}{8} m$.
We prove that the area of $A_2$
is given by an expression

\begin{equation} \label{D'}
|A_2| = \sum _{i=1} ^{m'} (-1)^{s_i} N(h_i) \frac{|M|}{N(h)} + O(m)\frac{|M|}{N(h)}
\end{equation}

where $s_i$ and $h_i$ are some integers and $m' \leq m$.
Note that $h_i$'s are allowed to coincide. 

Consider a connected component $L$ of $\partial A_2$.
Let $L_l= \sqcup l_i$ denote the union of all arcs of $L \setminus T$
with endpoints $a_1,...,a_{2n} \in T$. 
Each $a_{2i}$ is connected to $a_{2i+1}$ by a minimizing path in $T$.
Let $D'$ be one of the connected components of $M \setminus L$
(see Figure \ref{induction_step}).
Since the area of $A_2$ can be written as a sum 
$\sum \pm |D_i|$ (plus, possibly, the total area $|M|$), where each $D_i$ is a
complement of a connected component of  $\partial A_2$, it is enough to prove equation (\ref{D'}) for the
case when $A_2 = D'$.

\begin{figure} 
   \centering	
	\includegraphics[scale=0.6]{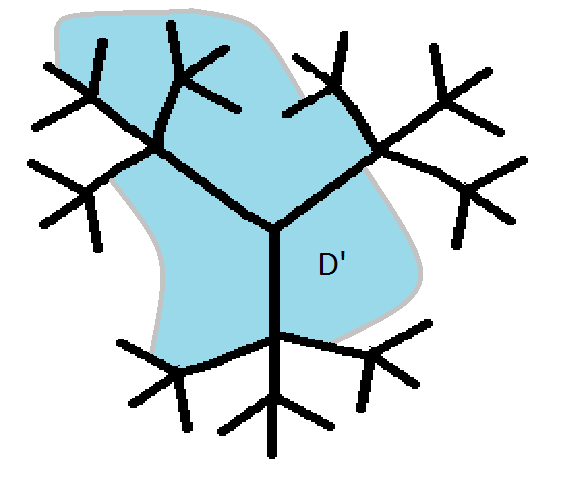}
	\caption{Calculating area of $D'$} \label{induction_step}
\end{figure}

We prove (\ref{D'}) by induction on the number of arcs of $L_l$.
Suppose first that $L_l$ consists of one arc with the same endpoints. 
Then $L_l$ intersects $T$ at one point $a$ dividing $T$ into two subtrees. 
Let $e$ be the edge of $T$ that contains $a$ and let $v$ be the vertex of $e$
that is further away from the root of $T$.
Then $v$ is the root of a ternary subtree $T'$ of height $h'$
that is either entirely inside $D'$ or outside $D'$.
Hence, by (\ref{long arc}), $|D'| = (N(h') + O(1)) \frac{|M|}{N(h)}$
or $|D'| = |M| -(N(h') + O(1)) \frac{|M|}{N(h)} = 
(N(h) - N(h') + O(1)) \frac{|M|}{N(h)} $.

Now suppose that $L_l$ consists of one arc with different endpoints $a_1$ and $a_2$.
As in the previous case we have one ternary subtree that is 
entirely inside or outside of $D'$. 
Let $t$ denote the shortest path in $T$ from $a_1$ to $a_2$ and let $k$
denote the number of whole edges in $t$. For each vertex $v$ in $t$ there 
are at most two edges starting at $v$ that are not in $t$. For each of them
there is a choice of whether they are inside or outside of $D'$.
Hence, we have $2 k$  more summands of the form
$\pm N(h_i) \frac{|M|}{N(h)}$. The area in the neighbourhood
of $t$ introduces a term of $\pm |k| \frac{|M|}{2N(h)} + O(\frac{|M|}{2N(h)}) $. 
We conclude that the area of $D'$ is given by 
$\sum _{i=1} ^{2(k+1)} (-1)^{s_i} N(h_i) \frac{|M|}{N(h)} + O(\frac{(k+n) |M|}{2N(h)})$

Suppose $L$ has $n$ arcs $l_1, ... , l_n$ 
with endpoints $a_1,...,a_{2n}$ and $t_1,...,t_n$ are geodesic paths in $T$
 from $a_{2i}$ to $a_{2i+1}$. Let $k_i$ denote the number of whole 
edges in $t_i$ and $k= \sum k_i$.
We claim that the area of $D'$ is given by
$|D'| = \sum _{i=1} ^{m'} (-1)^{s_i} N(h_i)\frac{|M|}{N(h)} + O( \frac{(k+n)|M|}{N(h)}) $,
where $m' \leq 2(n+k)$.

Suppose this is true for $n-1$ long arcs.

To prove inductive step consider $L$ with $n$ long arcs.
Let $l$ be one of them and let $D_l$ be the smaller disc of
$(M \setminus T) \setminus l$. We apply the base of induction to $D_l$ and inductive
assumption to $D' \setminus D_l$. If we cancel out the terms that come from 
the edges of $T$ that lie in the common boundary of $D_l$ and $D' \setminus D_l$
we obtain the desired expression.

This proves equation (\ref{D'}).

We claim that if the difference $|\sum _{i=1} ^{m'} (-1)^{s_i} N(h_i)\frac{|M|}{N(h)}
- \frac{|M|}{2}|$ is small then $m'$ must be large. This follows from a simple lemma below.

\begin{lemma} \label{sum}
Let $N > m > 0$ and $p \geq 3$ be positive integers. 
Let $\{ h_i \}_{i=0} ^{m-1}$
and $\{ s_i \}_{i=0} ^{m-1}$
be two collections of positive integers.
Then $|\frac{p^{N} -1}{p-1} - \sum _{i=0} ^{m-1} (-1)^{s_i}
 p^{h_i}|
\geq \frac{p^{N-m}-1}{p-1} $.
\end{lemma}

\begin{proof}

Fix positive integer $N$. The proof is by induction on m.
Let $m = 1$. For $s_0$ even, $h_0 = N - 1$ we get exactly 
$\frac{p^{N-1}-1}{p-1}$. For $s_0$ odd, $h_0 = N$
we obtain
$$p^N - \frac{p^N - 1}{p-1}=\frac{p^{N-1}(p^2 - 2p) +1}{p-1} > 
\frac{p^{N-1}+1}{p-1}$$
since $p > 1 + \sqrt{2}$.

For all other values of $h_0$ and $s_0$ we get  
$|\frac{p^{N} -1}{p-1} - (-1)^{s_0} p^{h_0}|>\frac{p^{N-1}-1}{p-1}$.

We assume that the statement of the lemma holds for $m - 1$ and we prove it for $m$.

Suppose, by contradiction, that 
$|\frac{p^{N} -1}{p-1} - \sum _{i=1} ^{m-1} (-1)^{s_i}
 p^{h_i} | < \frac{p^{N-m}-1}{p-1}$. Let $h_j$ be
the smallest power in the summation. If $h_j > N - m$
then $\sum _{i=1} ^{m-1} (-1)^{s_i} p^{h_i}$ is divisible by
$p^{N-m+1}$ and $\frac{p^{N} -1}{p-1}=
p^{N-m+1} (p^{m-2}+...+1) + \frac{p^{N-m+1}-1}{p-1}$.
Hence, $|\frac{p^{N} -1}{p-1} - \sum _{i=1} ^{m-1} (-1)^{s_i}
 p^{h_i} |$ is congruent modulo $p^{N-m+1}$ to $\frac{p^{N-m+1}-1}{p-1}$ or 
$p^{N-m+1}-\frac{p^{N-m+1}-1}{p-1}$. 
Since $p > 1 + \sqrt{2}$ we obtain $p^{N-m+1}-\frac{p^{N-m+1}-1}{p-1}
> \frac{p^{N-m+1}+1}{p-1} > \frac{p^{N-m+1}-1}{p-1}$.

If $h_j \leq N - m$ then $ |\frac{p^{N} -1}{p-1} - \sum _{i=1, i \neq j} ^{m-1} (-1)^{s_i}
 p^{h_i} | 
< \frac{p^{N-m}-1}{p-1} + p^{N-m} = \frac{p^{N-m+1}-1}{p-1} $,
which contradicts the inductive assumption.
\end{proof}


The difference between the volume of $A_2$ and half of the total volume is given by

$$|\frac{|M|}{2}-|A_2|| = |\frac{1}{2} N(h) - 
\sum _{i=1} ^{m'} (-1)^{s_i} N(h_i)|\frac{|M|}{N(h)} 
 + O(m) \frac{|M|}{N(h)}$$


Recall that $N(h)=\frac{3}{2}(3^h -1)$ and $m' \leq m \leq \frac{8}{3} |\partial A_1|$ 
(by inequality (\ref{A1A2})). We can rewrite the above expression as

 $$\frac{3}{2} |\frac{3^h - 1}{2}  - \sum _{i=1} ^{m'} (-1)^{s_i} 3^{h_i}| \frac{|M|}{N(h)} 
 + O(|\partial A_1|) \frac{|M|}{N(h)}$$

By Lemma \ref{sum} we obtain
$$|\frac{|M|}{2}-|A_2|| \geq \frac{3^{h-\frac{8}{3}|\partial A_1|} }{2}\frac{|M|}{N(h)} +O(|\partial A_1|) \frac{|M|}{N(h)}$$

Recall that $||A_2| - |A_1|| = O(|c|)\frac{|M|}{N(h)}$ by (\ref{A_2}).
As we are interested in the case when $h$ is large, we can assume $K>>4$, so $|Fill(c) - A_1| \leq \frac{4}{K}|c|<|c|$ 
and $|\partial A_1| \leq |c|$ by (\ref{A_1}). 
Collecting all the estimates together we get
$$\frac{|M|}{N(h)} > |\frac{|M|}{2}- Fill(c)|\geq |\frac{|M|}{2}-|A_2|| - |Fill(c) - |A_2||$$
$$\geq \frac{3^{h-\frac{8}{3}|c|}}{2} \frac{|M|}{N(h)} + O(|c|)\frac{|M|}{N(h)} $$

Cancelling out $\frac{|M|}{N(h)}$ on both sides and rearranging we obtain
$$\frac{8}{3} |c| + \log_3(O(|c|+1)) \geq h$$

Since $h$ can be made arbitrarily large it follows
that $|c|$ can be made arbitrarily large. This finishes the proof of Theorem \ref{main} in the case of sphere.

\section{Surfaces of large width}

The proof of Theorem \ref{main} for an arbitrary surface $M$ proceeds as follows.
We take the sphere constructed in the previous section
and glue in handles in a small disc away from the tree $T$ to get the desired topology.
As the handles are small they do not affect the diameter significantly.
However, since our $1-$cycle may wind around multiple times around the handles
we can not say that their contribution to the area of minimal filling will be small.
Because of this we will glue in handles in such a way that 
it does not change the total area of the surface and use the following lemma.

\begin{lemma}
Let $\Sigma$ be a compact Riemannian 2-surface with boundary and let 
$\Gamma$ denote the set of homologically trivial Lipschitz $1-$cycles
with coefficients in $\mathbb{Z}$ if $\Sigma$ is orientable and $ \mathbb{Z}_2$
otherwise.
Then $\inf \{ \frac{|\gamma|}{\sqrt{Fill(\gamma)}} | \gamma \in \Gamma \} > 0$.
\end{lemma}

\begin{proof}
Let $\gamma = \sum a_i f_i$ be a homologically trivial Lipschitz $1-$cycle
in $\Sigma$.
Without any loss of generality we can assume that each $f_i$
can be filled by a $2$-cycle of area $\leq |\Sigma|$.

Since $\Sigma$ is compact there exists a positive number $r>0$, such that
every ball $B_r \subset \Sigma$ of radius $\leq r$ is bilipschitz
homeomorphic to a subset of the positive half-plane $\mathbb{R}^2_+$
with bilipshitz constant $L=1+O(r^2)$.

Hence, if $|f_i| < r$ then it can be filled by a $2-$cycle of area 
$\leq (1+O(r^2)) \frac{|f_i|^2}{ \pi}$.
By choosing $r$ sufficiently small we obtain that for all $|f_i| < r$,
$\frac{|f_i|}{\sqrt{Fill(f_i)}}> \frac{ \sqrt{\pi}}{2}$.

On the other hand, we have $Fill(f_i) \leq |\Sigma|$ so 
for $|f_i|\geq r $ we have $\frac{|f_i|}{\sqrt{Fill(f_i)}} \geq \frac{2r}
{\sqrt{|\Sigma|}}$.

\end{proof}

Let $M_K$ denote the sphere from the previous section, where $K$ is the curvature
of hyperbolic disc $D$ used in the construction.
Let $O \in M_K$ be the point in the centre of the hyperbolic disc $D$ in $M_K$.
Choose a very small ball $U$ with a center at $O$. Since $U$ is very small it is 
$(1+\epsilon)-$bilipschitz diffeomorphic to a flat disc of the same radius.

Consider a subset $V$ of $M$ diffeomorphic to a disc. Define a Riemannian metric on
$V$ so that it is isometric to $M_K \setminus U$. We extend this metric to the
rest of $M$ in the following way. Let $\Sigma$ denote $M \setminus V$ and
$g$ denote its genus.
By considering embeddings of a disc or a Mobius strip in $\mathbb{R}^3$
and attaching handles it is easy to construct
a metric on $\Sigma$, so that

\begin{enumerate}
\item
$Area(\Sigma) = (g+10)$

\item
$\sup dist(x, \partial \Sigma) \leq 10$

\item
$|\partial \Sigma| = 1$
\end{enumerate}

We scale the metric so that $|\Sigma| = \frac{1}{2} |U|$.
Then $|\partial \Sigma|$ is approximately $\frac{|\partial U|}{\sqrt{2 \pi (g+10)}}$.
We attach a cylinder $C$ to $\Sigma$ of area $\frac{1}{2} Area(U)$ and lengths of boundary components
$|\partial \Sigma|$ and $|\partial U|$.
 Hence, we obtain a metric on $\Sigma$ that has
area and length of the boundary exactly the same as $U$.

Observe that up to scaling this construction of metric on $\Sigma$
 is independent of the curvature $K$ used in construction of $M_K$.
In particular, the scale invariant quantity $\alpha = \inf \{ \frac{|\gamma|}{\sqrt{Fill(\gamma)}} \}$,
where $\inf$ is taken over homologically trivial curves in $\Sigma$, does not depend on $K$.

Let $a$ and $b$ denote points on $\partial \Sigma$ and let $l_1(a,b)$ and $l_2(a,b)$ denote
lengths of minimizing geodesics between $a$ and $b$ in $\Sigma$ and $\partial \Sigma$ respectively.
Observe that infimum $\beta$ of the ratio $\frac{l_1}{l_2}$ is positive and scale-invariant.
In particular, it is also independent of $K$.

Now let $c$ be a cycle in $M$ and let $A$ be a nearly area-minimizing 
chain filling $c$ with $|\frac{|M|}{2}- |A||< \frac{|M|}{N}$.
If $c$ does not intersect $\partial \Sigma$ then since $|\Sigma|=|U|$
and $V$ is isometric to $M_K \setminus U$ we obtain that
$|c| \rightarrow \infty $ as $K \rightarrow \infty$.

Suppose $c \cap \Sigma$ is non-empty. If $c \subset \Sigma$
then $|A| \leq \alpha |c|^2 +\epsilon$ so $|c|$ has to be very large.

Suppose $c$ intersects $\partial \Sigma$ . We perturb $c$ so that it intersects 
$\partial \Sigma$ transversally. Let $a$ be an arc of $c \cap \Sigma$
and let $v_1,v_2 \in \partial \Sigma$ be endpoints of $a$. Let $l$
be the shorter arc of $\partial \Sigma$ from $v_1$ to $v_2$.

We can fill $a-l$ by a $2-$chain $B$ in $\Sigma$ of area
$|B| \leq (\frac{|a|+|l|}{\alpha})^2$ Adding $B$ to $A$ 
erases arc $a$ and substitutes it with $l$.
This changes length of $c$ by $||a|-|l||<(1+\frac{1}{\beta})|a|$.

Performing this procedure for each arc of $c$ in
$\Sigma$ we obtain a new $2-$chain $A_1$
with boundary $c_1$, such that $c_1 \cap \Sigma = \varnothing$.

We have $|c_1|= O(|c|)$ and $||A|-|A_1||=O(|c|^2)$.
Using our proof for the sphere we obtain

$$|\frac{|M|}{2}-|A_1|| \geq \frac{3^{h-\frac{8}{3}|c_1|}}{2} \frac{|M|}{N} + O(|c_1|)\frac{|M|}{N} $$ 

Thus, $3^{h-O(|c|)} \leq O(|c|^2)$. If follows that $O(|c|)\geq h$. This finishes the proof of the theorem.

\textbf{Proof of Theorem \ref{width}}

We will prove that every sweep-out of $M$ by $1-$cycles contains a cycle $c$,
such that $Fill(c)= \frac{|M|}{2}$. 

Suppose not, i.e. there exists a non-contractible loop $z$ in $Z_1(M)$, such that 
for every $c \in z$ we have $Fill(c) \leq \frac{|M|}{2} - \epsilon$.

Consider a very fine subdivision of $S^1$ $\{t_1,...t_n \}$.
Then $z(t_i)$ and $z(t_{i+1})$ are two cycles that are very close in flat norm.
We can fill $z(t_i)-z(t_{i+1})$ by a $2-$chain $A_i$ of area $\leq \frac{\epsilon}{4}$.
Proceeding this way we construct a $2$-cycle $A$. It follows
from the proof of Almgren's isomorphism theorem
 that for all sufficiently fine subdivisions of $S^1$
$A$ will represent a non-trivial element of $H_2(M, \Bbbk)$.

For each $z(t_i)$ we can fill it by a $2-$chain $B_i$ of area 
$\leq \frac{|M|}{2} - \frac{\epsilon}{4}$. 
Then $|B_i+A_i-B_{i+1}| < |M|$. Hence, this 2-cycle can be filled
by a $3-$chain. Proceeding this way we construct a filling
of $A$, which is a contradiction.

\end{document}